\documentclass{amsart}
\usepackage{hyperref}
\usepackage{times}
\usepackage{verbatim}
\usepackage[T1]{fontenc}
\usepackage[utf8]{inputenc}
\usepackage[english]{babel}
\usepackage{amssymb}
\usepackage{amsthm}
\usepackage{amsmath}
\usepackage{amstext}
\usepackage{multirow}
\usepackage{graphicx}
\usepackage[all]{xy}
\usepackage{mathrsfs}
\usepackage{xcolor}
\usepackage{fixfoot}
\newtheorem{thm}{Theorem}[section]

\newtheorem{lemma}[thm]{Lemma}
\newtheorem{prop}[thm]{Proposition}

\newtheorem{cor}[thm]{Corollary}

\newtheorem{example}[thm]{Example}

\newcommand{\mb}{\mathbb}

\newcommand{\C}{\mb C}

\newcommand{\Z}{\mb Z}

\newcommand{\OO}{\mathcal O}
\newcommand{\F}{\mathcal F}

\usepackage{mathtools}
\usepackage{ifmtarg}
\addtocounter{section}{0}             
\numberwithin{equation}{section}       

\title[Character and webs]{Character varieties, hexagonal 3-webs and Hess connection.}

\author[K. Diarra]{Karamoko Diarra}
\address{DER de Mathématiques et d’informatique, FAST, Université des Sciences, des Techniques et des
Technologies de Bamako, BP: E 3206 Mali.}
\email{karamoko.diarra2005@gmail.com}
\author[A. A. Diaw]{Adjaratou Arame Diaw}
\address{Univ Cheikh Anta Diop, Dakar, Senegal.}
\email{aramdiaw2@gmail.com}
\author[F. Loray]{Frank Loray}
\address{Univ Rennes, CNRS, IRMAR - UMR 6625, F-35000 Rennes, France.}
\email{frank.loray@univ-rennes.fr}
\author[B. Tangue Ndawa]{Bertuel Tangue Ndawa}
\address{University of Ngaoundéré,
University Institute of Technology, Department of Computer Engineering, N1, Dang, 455, Ngaoundéré, Cameroon.}
\email{bertuel.tangue@imsp-uac.org}

\thanks{The authors warmly thank Ludovic Rifford and Moustafa Fall for invitation at AIMS M'Bour Center, 
and Chaire Afrique CNRS for financial support. This work was done during our joint visit at M'Bour.
The three last authors also thank CIRM for hospitality, we started our collaboration there. 
The third author thanks  Universit\'e de Rennes 1, IRMAR, CNRS and France 2030 program Centre Henri Lebesgue ANR-11-LABX-0020-01 for constant support.}

\begin{document}
\begin{abstract} 
We investigate the existence of some affine structure on $\mathrm{SL}_2(\C)$-character varieties 
on the four-punctured sphere through web and symplectic geometry. This work provides a new proof
of non integrability (i.e. irreducibility) of Painlevé VI equations, previously proved by S. Cantat and the third author.
\end{abstract}
\maketitle

\sloppy
\tableofcontents

\section{Introduction}
In this paper, we consider the family of cubic hypersurfaces of $\C^3$:
\begin{equation}\label{eq:cubicsurface}
S_{\bold{a}}=\{(x_1,x_2,x_3)\in\C^3\ ;\ x_1^2+x_2^2+x_3^2+x_1 x_2 x_3=a_1 x_1+a_2 x_2+a_3 x_3+a_4\ \}
\end{equation}
where $\bold{a}=(a_1,a_2,a_3,a_4)\in\C$ are complex parameters. These surfaces appear as GIT quotients of the space
of representations $\mathrm{Hom}(\pi_1(\Sigma),\mathrm{SL}_2(\C))$ by the action of $\mathrm{SL}_2(\C)$
by conjugacy. See for instance \cite{BenedettoGoldman,Iwasaki,CantatLoray}. Here, parameters $A,B,C,D$ depend 
directly on conjugacy classes imposed at the punctures. These affine surfaces admit a group of polynomial
automorphisms generated (for generic parameters) by the three involutions
\begin{equation}\label{eq:involutions}
\left\{\begin{matrix}
\sigma_1:(x_1,x_2,x_3)\mapsto(a_1-x_1-x_2 x_3,\ x_2,\ x_3)\\
\sigma_2:(x_1,x_2,x_3)\mapsto(x_1,\ a_2-x_2-x_1 x_3,\ x_3)\\
\sigma_3:(x_1,x_2,x_3)\mapsto(x_1,\ x_2,\ a_3-x_3-x_1 x_2)
\end{matrix}\right.
\end{equation}
This group $\Gamma=\langle \sigma_1,\sigma_2,\sigma_3 \rangle$ generated is the free product of these involutions, 
i.e. there are no other relations between them.
There is a holomorphic symplectic structure (or volume form) defined on $S_{\bold{a}}$ by the holomorphic $2$-form
\begin{equation}\label{eq:symplectic form}
\Omega:=\frac{dx_1\wedge dx_2}{2x_3+x_1 x_2-a_3}=\frac{dx_2\wedge dx_3}{2x_1+x_2 x_3-a_1}
=\frac{dx_3\wedge dx_1}{2x_2+x_1 x_3-a_2};
\end{equation}
here, the equalities can be deduced by taking exterior product of $dx_i$'s with the differential of the equation of $S_{\bold{a}}$. We note that $(\sigma_i)^*\Omega=-\Omega$ for $i=1,2,3$
so that the $2$-index subgroup $\Gamma_2=\langle \sigma_1\circ\sigma_2,\sigma_2\circ\sigma_3 \rangle$ is symplectic,
i.e. preserving $\Omega$. We note that $\Gamma_2$ is the free-group
with two generators, and it can be viewed as the monodromy of the Painlevé VI equation (or foliation), 
see \cite{DubrovinMazzocco,InabaIwasakiSaito,CantatLoray}. This relationship is used in \cite{CantatLoray}
to prove irreducibility of Painlevé VI equations for generic parameters. The motivation of our work (see section \ref{sec:motivation})
is to simplify the end of the proof of the main Theorem of \cite{CantatLoray}.

In this paper, we consider the geometry of the three foliations given by the level curves of coordinates
\begin{equation}\label{eq:feuilletagesconic}
\F_1=\ker(dx_1),\ \ \ \F_2=\ker(dx_2),\ \ \ \text{and}\ \ \ \F_3=\ker(dx_3)
\end{equation}
in restriction to the surface $S_{\bold{a}}$. Our first result is about the nature of the $3$-web 
induced on the smooth part of the surface. Recall that a $3$-web is parallelizable if there exist
local coordinates through which the $3$-web is is defined by $3$ constant foliations (i.e. by parallel lines).
In general, a $3$-web is not parallelizable and one can define a curvature, a $2$-form, which measures
the obstruction to be parallelizable. We succeed to compute this $2$-form and prove

\begin{thm}\label{thm:3webpara} The $3$-web $(\F_1,\F_2,\F_3)$ on $S_{\bold{a}}$ is 
locally parallelizable if, and only if, $\bold{a}=(0,0,0,4)$.
\end{thm}

In the second part of our paper, we consider only two among the foliations, say  $(\F_1,\F_2)$,
at the neighborhood of a generic point of $S_{\bold{a}}$ where they are transversal.
Together with the symplectic form $\Omega$, they locally define a bi-Langrangian structure
on $S_{\bold{a}}$. Following H. Hess and M. N. Boyom, one can associate a torsion-free connection
$\nabla$ on the tangent bundle $T_{S_{\bold{a}}}$, which is the unique one that parallelizes $\Omega$
and preserves $\F_1$ and $\F_2$. The flatness of this connection is given by 
the vanishing of its curvature, namely $\nabla\cdot\nabla$.
We succeed to compute the curvature of this connection, and we prove:

\begin{thm}\label{thm:bilagrang} The bi-Lagrangian manifold $(S_{\bold{a}},\Omega,\F_1,\F_2)$
has flat Hess connection if, and only if, $\bold{a}=(0,0,0,4)$.
\end{thm}

In \cite{vdPutSaito}, M. van der Put and M.-H. Saito construct the character varieties associated
to other Painlevé equations I-V. They are also given by cubic affine surfaces $S=\{P(x_1,x_2,x_3)=0\}\subset\C^3$
where the list of polynomials $P$ (with parameters) is given in (\ref{eq:ListCharaterPainleveSvdP}).
They are a kind of degenerate versions of the family $S_{\bold{a}}$. They also admit a natural symplectic 
structure (\ref{eq:SymplecticPVtoPI}) which coincides with the more general one defined by P. Boalch in \cite{Boalch}.
We can define (singular) foliations $\F_i:=\ker{dx_i}$ as before and address the same questions to these surfaces.
We prove:

\begin{thm} For any of the cubic surfaces $S$ in the list (\ref{eq:ListCharaterPainleveSvdP}),
we have:
\begin{itemize}
\item the $3$-web $(\F_1,\F_2,\F_3)$ is not parallelizable except for $PII^{FN},PII,PI$;
\item for at least one bi-Lagrangian manifold $(S,\Omega,\F_i,\F_j)$
defined by two among three foliations $\F_1,\F_2,\F_3$,
the corresponding Hess connection is not flat.
\end{itemize}
\end{thm}

Of course, this is only heuristic, experimental, as the general strategy of \cite{CantatLoray} to irreducibility 
of Painlevé equations (in particular Theorem \ref{lem:parallelizable}) fails for Painlevé equation $PI$ to $PV$
because of the lack of dynamics so far. It would be interesting to investigate this direction with the recent
works of \cite{Klimes,PaulRamis1,PaulRamis2}. On the other hand, for higher dimensional character variaties,
W. Goldman defined analogues of our foliations $\F_i$ which turn to be Lagrangian for the natural 
symplectic structure. Then it would be intersting to address the questions of this article to this more general setting.

\section{Motivations and applications: Painlevé VI equations}\label{sec:motivation}

The irreducibility of Painlevé equations means, roughly speaking, that their integration cannot reduce 
as  into successive integrations of linear and/or first order non linear differential equations. This has been
proved for Painlevé I equation  by K. Nishioka and H. Umemura independantly using non linear differential
Galois theory. Another approach has been developped by G. Casale in \cite{CasalePainleve1,CasaleNishioka}
using B. Malgrange's Galois groupoid for foliations \cite{Malgrange}. Casale proves in \cite{CasaleNishioka}
that, if the Galois groupoid of a Painlevé foliation is the full symplectic pseudo-group, then the corresponding 
Painlevé equation is irreducible, in the sense of Nishioka-Umemura above. By this way, Casale provides 
a new proof of irreducibility of Painlevé I equation in \cite{CasalePainleve1}. In order to prove that the Galois
groupoid is the full symplectic pseudo-group, it suffices to prove that the Malgrange-Galois groupoid does not
preserve web or symplectic affine structure, thanks to \'E. Cartan's classification of Lie pseudo-groups in dimension $2$. 
In the case of Painlevé VI, the Malgrange-Galois groupoid contains at least the monodromy of the foliation.
Therefore, the proof of the irreducibility of a Painlevé VI equation reduces to exclude the existence of a web or affine
structure on the character variety $S_{\bold{a}}$ that is invariant under the monodromy group $\Gamma_2$.
This is how irreducibility is proved for most of Painlevé VI equations in \cite{CantatLoray}. Namely, it is proved
that 

\begin{thm}\cite[Prop. 5.3, Th. 5.4]{CantatLoray}\label{thm:NoAffine} The dynamical system $\left(S_{\bold{a}},\Gamma_2\right)$ admits no invariant holomorphic web
structure, whatever is the value of $\bold{a}$. Moreover, it admits an invariant affine structure if, and only if,
$\bold{a}=(0,0,0,4)$.
\end{thm}

Here, by symplectic affine structure, we mean an atlas of holomorphic charts $\{\phi_i:U_i\hookrightarrow\C^2\}$ 
on the smooth part of $S_{\bold{a}}$ sending $\Omega$ to the standart $2$-form $du\wedge dv$ in $\C^2$,
with transition charts $\phi_i\circ\phi_j^{-1}$ belonging to the pseudo-group induced by the symplectic affine group
\begin{equation}\label{eq:affinegroup}
\mathrm{Aff^1}:=\left\{\left.\phi:
\begin{matrix}\C^2&\to&\C^2\\ (u,v)&\mapsto& (au+bv+t,cu+dv+s)\end{matrix}\ \right|
\begin{matrix}a,b,c,d,t,s\in\C,\\ ad-bc=1\end{matrix}
\right\}
\end{equation}
The $\Gamma_2$-invariance means that transformations $\gamma\in\Gamma_2$ become affine transformations
through the affine atlas. 

In the case $\bold{a}=(0,0,0,4)$, the cubic surface is the quotient $\Pi:\C^2\to S_{(0,0,0,4)}$
of $\C^2$ by the group $H=\{(u,v)\mapsto(\pm u+t,\pm v+s)\}$, and the action of $\Gamma_2$
is induced by the group $\mathrm{SL}_2(\Z)$ on the quotient (up to finite index).
So $\left(S_{(0,0,0,4)},\Gamma_2\right)$ inherits a affine structure: affine charts are local inverses of $\Pi$.
The corresponding Painlevé VI equations have special first integrals as shown in \cite{CasalePicard}.

In order to prove Theorem \ref{thm:NoAffine}, one can first easily prove the following 

\begin{thm}\label{lem:parallelizable} 
Assume that a (non empty) Zariski open set $S^o\subset S_{\bold{a}}$ 
admits an affine structure which is invariant under a finite index subgroup $G<\Gamma_2$.
Then, the $3$-web $(\F_1,\F_2,\F_3)$ is parallel through the affine charts.
\end{thm}

\begin{proof} The problem being symmetric with respect to variables $x_1,x_2,x_3$, 
it is enough to prove that $\F_3$ is parallel in charts. First of all, notice that this does not depend 
on the choice of affine chart since affine transformations preserve parallel lines. Moreover, by analytic continuity,
it is enough to prove the statement for a single local chart. The automorphism $\phi=\sigma_1\circ\sigma_2$
is preserving the foliation $\F_3$. Let us recall its dynamics on $S_{\bold{a}}$ (see \cite[section 5.1]{CantatLoray}).
It can be written as
\begin{equation}\label{eq:phiz}
\phi\ :\ \begin{pmatrix}x_1\\ x_2\end{pmatrix}\mapsto \begin{pmatrix}-1&-x_3\\ x_3&x_3^2-1\end{pmatrix}\begin{pmatrix}x_1\\ x_2\end{pmatrix}+\begin{pmatrix}a_1\\ a_2-a_1 x_3\end{pmatrix}
\end{equation}
We note that, in restriction to $x_3=t$, for a fixed scalar $t\in\C$, the matrix part has determinant $1$ (i.e. it lies in $\mathrm{SL}_2(\C)$) and trace $t^2-2$. Therefore, when $t=\pm 2\cos(\pi\theta)$ with $\theta=\frac{p}{q}$ is rational, with $p,q\in\Z$ relatively prime,
then $\phi_t$ is periodic of order $2q$ in restriction to the curve $C_t:=S_{\bold{a}}\cap\{x_3=t\}$ (a leaf of $\F_3$).
On the other hand, $\phi$ has infinite order in restriction to the curves  $C_{t}$ for generic $t$.
Now, since the group $G$ has finite index in $\Gamma_2$, then $G$ must contain an iterate $g=\phi\circ\cdots\circ\phi$ 
of $\phi$.
For infinitely many $t$ (i.e. corresponding to rational $\theta$), there is an iterate $g^N$ on $S_{\bold{a}}$, 
$N\in\Z\setminus\{0\}$,  which 
fixes point-wise $C_{t}$ and is not the identity. The union of all those curves $C_{t}$ is Zariski dense and intersect
the Zariski-open set on which the symplectic affine structure is well-defined. 
One can find an affine chart $\{\phi_0:U_0\hookrightarrow\C^2\}$ and an infinite sequence $(N_n,t_n)$, $n\in\Z_{>0}$
where $C_{t_n}$ is fixed point-wise by $g^{N_n}$, and the $C_{t_n}$'s are pair-wise distinct, intersecting the chart $U_0$.
Through the coordinate chart $\phi_0$, the curve $C_{t_n}$ is transformed into a smooth analytic curve $C_{t_n}^0$
and the transformation $g^{N_n}$ induces a (non trivial) symplectic affine transformation, fixing $C_{t_n}^0$ point-wise.
After conjugating by an affine transformation, we can assume that $g^{N_n}$ is a linear transformation. 
Then, the fixed point set $C_{t_n}^0$ must correspond to an eigenspace associated to eigenvalue $1$, and be 
is linear, i.e. a line.

Therefore, the image $\F_3^0$ of $\F_3$ by the affine chart $\phi_0$ has the property
that infinitely many leaves are lines, namely those $C_{t_n}^0$. Moreover, by density of rational $\theta$'s among real ones, 
one can insure to have accumulation of those lines/leaves inside the chart $\phi(U_0)$. {\it We claim} that all leaves must be lines.
Indeed, we can define the (holomorphic) curvature of the leaves as follows.
Consider a local holomorphic first integral $f(u,v)$ for $\F_3^0$ in the affine chart; 
maybe changing by an affine transformation,
we can assume ${\partial f}{\partial v}\not=0$.
Then, a leaf is locally defined by $v=\varphi(u)$ for some implicit solution of $f(u,\varphi(u))=0$.
By derivation, we get
$$\frac{\partial f}{\partial u}(u,\varphi(u))+\frac{\partial f}{\partial v}(u,\varphi(u))\cdot \varphi'(u)=0$$
from which we can deduce the derivative at a point $(u,\varphi(u))$:
$$\varphi'(u)=-\frac{f_u}{f_v}$$
(with usual notation $f_u,f_v$ for partial derivatives). After derivating again, we get
$$f_{uu}+2 f_{uv}\cdot \varphi'+f_{vv}\cdot (\varphi')^2+f_v \cdot \varphi''$$
at the point $(u,\varphi(u))$ which gives
$$\varphi''=\underbrace{\frac{-f_{u}^2f_{vv}+2f_{u}f_{v}f_{uv}-f_{v}^2f_{uu}}{f_v^3}}_{\gamma(u,v)}.$$
This holomorphic ``curvature'' function $\gamma(u,v)$ is therefore measuring the second derivating
of the local leaf, and is vanishing identically along all leaves which are lines.
But the latter subset is Zariski dense as mentionned before. 
We already conclude that $\F_3^0$ is a foliation by lines, as claimed before. 

Now, the symplectic affine transformation induced by $g^{N_n}$ is preserving all leaves of $\F_3^0$,
and fixing one point-wise, namely $C_{t_n}^0$. We can modify the affine chart $(u,v)$ such that 
$C_{t_n}^0$ is given by $v=0$, so that $g^{N_n}$ is linear, and takes the form
\begin{equation}\label{eq:phiziterateaffine}
\phi\ :\ (u,v)\mapsto(u+tv,v)
\end{equation}
for some $t\in\C$ (because we have $\phi(u,0)=u$ and $\det(\phi)=1$).
Note that $g^{N_n}$ is not the identity and $t\not=0$.
But the only invariant lines of this transformation are horizontal line, and we conclude that $\F_3^0$ is 
the horizontal foliation $\ker(dv)$, i.e. a foliation by parallel lines, which does not depend on the choosen affine chart.
\end{proof}

We promptly conclude that, under assumptions of Theorem \ref{lem:parallelizable},
the $3$-web is parallelizable. Therefore, Theorem \ref{thm:3webpara} immediately implies Theorem \ref{thm:NoAffine}.

In a similar way, the bi-Lagrangian manifold $(S_{\bold{a}},\Omega,\F_1,\F_2)$
is locally equivalent through affine charts to $(\C^2,du\wedge dv,\ker(du),\ker(dv))$, 
under conclusion of Theorem \ref{lem:parallelizable}, and we promptly deduce that its Hess connection is flat.
Therefore, Theorem \ref{thm:bilagrang} immediately implies Theorem \ref{thm:NoAffine}.

\begin{cor} Assume that a (non empty) Zariski open set $S^o\subset S_{\bold{a}}$ 
admits an affine structure which is invariant under a finite index subgroup $G<\Gamma_2$.
Then, the $3$-web $(\F_1,\F_2,\F_3)$ is parallelizable through the affine charts.
\end{cor}

\section{Parallelizable webs}\label{sec:Webs}

An analytic (singular) $k$-web on a surface, $k\in\mathbb Z_{>1}$, is locally defined (see \cite{PereiraPirio}) 
by a differential equation
$F(x,y,y')=0$ where $F(x,y,z)\in\C\{x,y\}[z]$ is a reduced polynomial of degree $k$ in variable $z$,
with analytic coefficients in $(x,y)$. The leaves of the web are parametrized curves $x\mapsto (x,y(x))$
for solutions $y(x)$ of the differential equation. There are $k$ distinct leaves passing through a general point.
Precisely, since $F$ is reduced, the discriminant of $F$ with respect to $z$
is not identically vanishing, and its zero set $\Sigma$ is a (possibly empty) hypersurface. 
It is the singular set of the web. 
Outside of $\Sigma$, the web locally splits as the superposition of $k$ pairwise transversal foliations
$\F_1,\ldots,\F_k$. A regular $2$-web can be locally straightened to the superposition 
of foliations by parallel lines, e.g. $(y')^2=1$. However, for $k>2$, there are local obstruction to straigthen $k$-webs
to superposition of foliations by parallel lines. Let us consider the case of $3$-webs that we are concerned with.

Let $(\F_1,\F_2,\F_3)$ be a locally decomposed $3$-web, where 
$\F_i$ is the foliation defined by $y'=f_i(x,y)$, where $f_i(0,0)\not= f_j(0,0)$ for $i\not=j$.
We can rewrite the differential equation $dy-f_i(x,y)dx=0$, i.e. as defined by the kernel $\ker(\omega_i)$ of the $1$-form 
$\omega_i=dy-f_i(x,y)dx$. It will be usefull to define a foliation as kernel $\F=\ker(\omega)$ by a general 
non vanishing analytic $1$-form $\omega=a(x,y)dx+b(x,y)dy$: if $b(0,0)\not=0$, then we deduce the differential equation
$\frac{dy}{dx}+\frac{a(x,y)}{b(x,y)}=0$; if $b(0,0)=0$, then necessarily $a(0,0)\not=0$, and we get the differential equation
$\frac{dx}{dy}+\frac{b(x,y)}{a(x,y)}=0$. In other words, the $1$-form $\omega$ defining the foliation $\F$
is not uniquely defined by $\F$, but is uniquely defined up to the multiplication by a non vanishing function.
The following study is standard, and goes back to Blaschke for instance.

\begin{lemma}\label{lem:normalforms3webs} We can locally define the regular $3$-web 
$\mathcal W:=(\F_1,\F_2,\F_3)$ by $1$-forms
$\F_i=\ker(\omega_i)$ satisfying 
\begin{equation}\label{eq:sumzero}
\omega_1+\omega_2+\omega_3=0.
\end{equation}
Moreover, any other triple of $1$-forms $(\tilde\omega_i)_i$ defining the same $3$-web and satisfying relation (\ref{eq:sumzero})
writes $\tilde\omega_i=f(x,y)\omega_i$ for a (common) non vanishing function $f$.
\end{lemma}

\begin{proof} Choose $\omega_i=a_i(x,y)dx+b_i(x,y)dy$ non vanishing $1$-form defining $\F_i$
and let us see how to modify them to get the expected linear dependance.
Since the $3$-web is regular, we get that $\omega_i\wedge\omega_j$ is not vanishing (transversality 
between $\F_i$ and $\F_j$), and therefore 
\begin{equation}
\delta_{i,j}:=\det\begin{pmatrix}a_i & a_j\\ b_i & b_j\end{pmatrix}\not=0.
\end{equation}
One can then write 
\begin{equation}
\underbrace{\delta_{2,3}\omega_1}_{\tilde\omega_1}+\underbrace{\delta_{3,1}\omega_2}_{\tilde\omega_2}+\underbrace{\delta_{1,2}\omega_3}_{\tilde\omega_3}=0
\end{equation}
and get new defining $1$-forms $\tilde\omega_i$ satisfying the given relation. 

Once we get a triple $(\omega_1,\omega_2,\omega_3)$ defining $\mathcal W$ and satisfying (\ref{eq:sumzero}),
then for any non vanishing function $f(x,y)$, the multiple $(f\omega_1,f\omega_2,f\omega_3)$ also defines $\mathcal W$ and satisfies (\ref{eq:sumzero}). Conversely, given any other triple $(f_1\omega_1,f_2\omega_2,f_3\omega_3)$ 
defining $\mathcal W$ and satisfying (\ref{eq:sumzero}), after substituting $\omega_3=-\omega_1-\omega_2$,
we get
$$(f_1-f_3)\omega_1+(f_2-f_3)\omega_2=0$$
which, by transversality, implies 
$$f_1-f_3=f_2-f_3=0$$
and therefore $f_1=f_2=f_3=:f$. 
\end{proof}

\begin{lemma}\label{lem:theta}
Let $(\omega_1,\omega_2,\omega_3)$ satisfying (\ref{eq:sumzero}) as in Lemma \ref{lem:normalforms3webs}.
Then, there exists a unique $1$-form $\theta$ satisfying $d\omega_i=\theta\wedge\omega_i$ for $i=1,2,3$.
Moreover, if we change $(\omega_i)$ by $(f\omega_i)$ for some non vanishing analytic function $f$,
then we change $\theta$ by $\theta+\frac{df}{f}$.
\end{lemma}

\begin{proof}Like in the proof of Lemma \ref{lem:normalforms3webs}, one can determine 
$\theta$ from relations $d\omega_i=\theta\wedge\omega_i$ for $i=1,2$ by linear algebra.
Then, by relation (\ref{eq:sumzero}), we automatically get $d\omega_3=\theta\wedge\omega_3$.
For the last part of the statement, notice that 
\begin{equation}\label{eq:chtfacteurtheta}
d\omega=\theta\wedge\omega\ \ \ \Leftrightarrow\ \ \ d(f\omega)=\left(\theta+\frac{df}{f}\right)\wedge(f\omega).
\end{equation}
\end{proof}

\begin{cor}The analytic $2$-form $\kappa:=d\theta$ only depends on the web $\mathcal W:=(\F_1,\F_2,\F_3)$,
does not depend on the choice of the triple $(\omega_1,\omega_2,\omega_3)$.
\end{cor}

\begin{proof}Just observe that $d\left(\theta+\frac{df}{f}\right)=d\theta$.
\end{proof}

We call $\kappa$ the Blaschke curvature of the $3$-web $\mathcal W$.

\begin{thm}\label{lem:Blaschkecurvature2form}
Let $\mathcal W:=(\F_1,\F_2,\F_3)$ be a regular $3$-web and $\kappa$
its Blaschke curvature. Then  $\mathcal W$ is locally parallelizable if, and only if, $\kappa=0$
(i.e. $\kappa$ is the zero $2$-form).
\end{thm}

Here, parallelizable means that, after a change of coordinates, each foliation is defined by a constant 
$1$-form $\omega_i=a_idx+b_idy$, where $a_i,b_i\in\C$, and $a_ib_j-a_jb_i\not=0$ for $i\not=j$.
Observe that all parallel $3$-webs are equivalent by linear change of coordinates, i.e. equivalent 
for instant to the web defined by $(dx,dy,-dx-dy)$. 

\begin{proof}[Proof of Theorem \ref{lem:Blaschkecurvature2form}]
If $\mathcal W$ is parallelizable, then we can compute the $2$-form $\kappa$
in a system of coordinates where $\mathcal W$ is parallel, say defined by $(dx,dy,-dx-dy)$.
Then, clearly, $\theta=0$ and $\kappa=d\theta=0$.

Conversely, if $\kappa=0$, then one can write $\theta=\frac{df}{f}$ for some non vanishing analytic function $f$.
Then replacing $\omega_i$ by $\tilde\omega_i=\frac{\omega_i}{f}$, we get $d\tilde\omega_i=0$.
Therefore, $\tilde\omega_i=dh_i$ for some analytic function $h_i$, $i=1,2,3$, and choosing the constant 
allow use to insure that $h_1+h_2+h_3=0$. Finally, the change of coordinate $\phi(x,y):=(h_1(x,y),h_2(x,y))$
is sending $\mathcal W$ to the parallel $3$-web defined by $(dx,dy,-dx-dy)$.
\end{proof}

Given a general regular $3$-web $\mathcal W:=(\F_1,\F_2,\F_3)$, 
one can always parallelize $(\F_1,\F_2)$. In other words, we can assume that the $3$-web 
is defined by $(dx,dy, a(x,y)dx+b(x,y)dy)$. In this situation, the Blaschke curvature reads as follows

\begin{prop}\label{prop:computecurvdxdy}
The Blaschke curvature of the $3$-web defined by $(dx,dy, a(x,y)dx+b(x,y)dy)$ is given by
\begin{equation}\label{eq:BlaschkeCurvdxdy}
\kappa=\left(\frac{a_{xy}a-a_x a_y}{a^2}-\frac{b_{xy}b-b_x b_y}{b^2}\right) dx\wedge dy
\end{equation}
(where subscripts $f_x,f_y,f_{xy}$ stand for $\frac{df}{dx},\frac{df}{dy},\frac{d^2f}{dxdy}$).
\end{prop}

\begin{proof} A normalized triple satisfying (\ref{eq:sumzero}) is given by
$(-adx,-bdy, adx+bdy)$. Then $\theta=\frac{b_x}{b}dx+\frac{a_y}{a}dy$
and we get the result.
\end{proof}

\begin{proof}[Proof of Theorem \ref{thm:bilagrang}]
Let us now apply this study to the $3$-web induced by $(dx_1,dx_2,dx_3)$ on the affine surface
$S_{\bold{a}}$ of the introduction, defined as the zero-set in $\C^3$ of the polynomial
\begin{equation}\label{eq:P}
P(x_1,x_2,x_3)=x_1^2+x_2^2+x_3^2+x_1 x_2 x_3-(a_1 x_1+a_2 x_2+a_3 x_3+a_4).
\end{equation}
Through the projection 
$$S_{\bold{a}}\to \C^2\ ;\ (x_1,x_2,x_3)\mapsto(x_1,x_2)$$
the $3$-web projects to:
$$(dx_1,dx_2,P_{x_1}dx_1+P_{x_2}dx_2)$$
where $P_{x_1},P_{x_2}$ stand for 
partial derivatives with respect to $x_1,x_2$ respectively.
Then, setting $a=2x_1+x_2 x_3-a_1$ and $b=2x_2+x_1 x_3-a_2$ in formula (\ref{eq:BlaschkeCurvdxdy}),
and computing partial derivatives as follows
\begin{equation}\label{eq:partialderivative}
\frac{\partial}{\partial x}=\frac{\partial}{\partial x_1}-\frac{P_{x_1}}{P_{x_3}}\frac{\partial}{\partial x_3}\ \ \ \text{and}\ \ \ 
\frac{\partial}{\partial y}=\frac{\partial}{\partial x_2}-\frac{P_{x_2}}{P_{x_3}}\frac{\partial}{\partial x_3},
\end{equation}
we get a large expression (for instance using a symbolic computation software), a rational function of $(x_1,x_2,x_3)$.
Its numerator is a polynomial of degree $6$ in variable $x_3$. Then, using relation $P(x_1,x_2,x_3)=0$ on the surface, 
one can reduce it as a polynomial
of degree $1$ in $x^3$. Then, the curvature is vanishing identically if, and only if, all coefficients of this later polynomial
in variables $(x_1,x_2,x_3)$ vanish. This leads to a system of polynomial equations which reduces, via a Gr\"obner basis,
to $a_1=a_2=a_3=a_4-4=0$. This proves Theorem \ref{thm:3webpara}. This proves Theorem \ref{thm:bilagrang}.
\end{proof}

\section{Flatness of the Hess connection}\label{sec:Hess}



Let $(M,\omega)$ be  a symplectic manifold. 
A foliation $\F$ on $M$ is said Lagrangian if its leaves are locally Lagrangian submanifolds, i.e. if
\begin{equation}\label{eq:Lagrangian1}
\omega(X,Y)=0\ \ \ \text{for any}\ \ \ X,Y\in\Gamma\left(\F\right).
\end{equation}
and
\begin{equation}\label{eq:Lagrangian2}
\dim(\F)=\frac{\dim(M)}{2},
\end{equation}
which is the maximal dimension for a foliation (or submanifold) satisfying (\ref{eq:Lagrangian1}).

A bi-Lagrangian structure on $(M,\omega)$ is the data of a pair $(\F_{1},\F_{2})$ of Lagrangian foliations on $(M,\omega)$ 
which are transversal  to each other; the quadruplet  $(M,\omega,\F_1,\F_2)$ is called a bi-Lagrangian manifold.

\begin{example}\label{ex:TrivialBiLagrangian}
On sufficiently small neighborhood $U\subset M$ of any point $p\in M$,
there exist Darboux coordinates $(x_1\ldots,x_n,y_1,\ldots,y_n)$ for the symplectic structure,
i.e. where the symplectic $2$-form writes $\omega\vert_U=dx_1\wedge dy_1+\ldots+dx_n\wedge dy_n$.
Then, the two foliations $\F_{dx}:=\ker(dx_1\wedge\cdots\wedge dx_n)$ and $\F_{dy}:=\ker(dy_1\wedge\cdots\wedge dy_n)$ 
define a bi-Lagrangian structure. In general, a bi-Lagrangian structure need not be given in a so simple form locally,
as we shall see.
\end{example}

Let $\nabla$ be an affine connection on $M$, i.e. a map
$$
\begin{matrix}
\Gamma(TM)\times\Gamma(TM)&\to& \Gamma(TM)\\
(X,Y)&\mapsto&\nabla_XY
\end{matrix}
$$
which is $\OO$-linear with respect to $X$, $\C$-linear with respect to $Y$,
and satisfying the Leibniz rule:
$$ \nabla_X(gY)=dg(X)\cdot Y+g\cdot\nabla_XY.$$
We call torsion the $\OO$-linear map
$$  T^\nabla\ :\ (X,Y)\mapsto \nabla_XY-\nabla_YX-[X,Y] $$
and we say that $\nabla$ is torsion-free if this map is vanishing identically.

We say that $\nabla$ parallelizes the symplectic form $\omega$ if
\begin{equation} \label{Bieq3}
\omega(\nabla_{X}{ Y},Z)+\omega(Y,\nabla_{ X}{Z})=X\cdot \omega(Y,Z)   
\end{equation} 
for any $X,Y,Z\in\Gamma(TM)$. We say that $\nabla$ preserves a foliation $\F$ on $M$ if
\begin{equation}  \label{Bieq4}
\nabla_XY\in\Gamma\left(\mathcal{F}\right)   
\end{equation}	 
for any $X\in\Gamma(TM)$ and $Y\in\Gamma(F)$.
Then, Hess proved the following result:

\begin{thm}\cite[Theorem 1]{Hess}\label{b20}
	Let	$(M,\omega,\F_1,\F_2)$ be a bi-Lagrangian manifold. Then, there exists  a unique torsion-free connection $\nabla$ on $M$ which parallelizes $\omega$, and preserves both foliations $\F_1,\F_2$.  
\end{thm}

The connection given by Theorem~\ref{b20} is called Hess connection of the bi-Lagrangian structure.
Theorem \ref{b20} holds also for holomorphic symplectic structure and we can define Hess connection as well.
As a matter of example, we have the following:

\begin{prop}\label{prop:HessConnection}
In $\C^2$, let $\omega=f(x,y)dx\wedge dy$ for an holomorphic function $f$. 
Then we have a bi-Lagrangian manifold $(\C^2,\omega,\F_{dx},\F_{dy})$, whose Hess connection satisfies
\begin{equation}\label{eq:HessConnection}
\nabla_{\partial_x}\partial_x=\frac{f_x}{f}\cdot\partial_x,\ \ \ 
\nabla_{\partial_y}\partial_y=\frac{f_y}{f}\cdot\partial_y,\ \ \ 
\text{and}\ \ \ \nabla_{\partial_x}\partial_y=\nabla_{\partial_y}\partial_x=0.
\end{equation}
where $f_x$ and $f_y$ stand for partial derivatives $\partial_x(f)$ and $\partial_y(f)$ respectively.
\end{prop}

We can rewrite the $\C$-linear map $Y\mapsto\nabla_XY$ 
in the $\OO$-basis $<\partial_x,\partial_y>$ of $T\C^2$:
\begin{equation}\label{eq:HessMatrixConnection}
\nabla_XY=i_X(dY+\Omega Y)\ \ \ \text{where}\ \ \ \Omega= \begin{pmatrix}
\frac{f_x}{f}dx & 0\\0&\frac{f_y}{f}dy
\end{pmatrix}
\end{equation}
for all $X,Y\in\Gamma(T\C^2)$. Before proving Proposition \ref{prop:HessConnection},
let us discuss its integrability.

We say that the connection $\nabla$ is flat (or integrable) when $\nabla\cdot\nabla=0$
which is equivalent to the vanishing of the curvature tensor
$$ R^\nabla\ :\ (X,Y,Z)\mapsto \nabla_X \nabla_Y Z- \nabla_Y \nabla_X Z-\nabla_{[X,Y]}Z$$
which is equivent, in matrix form, to $d\Omega+\Omega\wedge\Omega=0$.
We promptly deduce:

\begin{cor}\label{cor:FlatHess}
The Hess connection of the bi-Lagrangian manifold $(\C^2,\omega,\F_{dx},\F_{dy})$
defined in Proposition \ref{prop:HessConnection} by $\omega=f(x,y)dx\wedge dy$ is flat if, and only if,
\begin{equation}\label{eq:FlatHess}
\frac{\partial^2f}{dx\ dy}\cdot f=\frac{\partial f}{dx}\cdot \frac{\partial f}{dy}.
\end{equation}
\end{cor}

\begin{proof}We have:
\begin{equation}\label{Weq5}
d\Omega+\underbrace{\Omega\wedge\Omega}_{=0}= d\Omega=
\begin{pmatrix}
- 1& 0\\0&1
\end{pmatrix}\cdot \frac{f_{xy}f-f_x\cdot f_y}{f^2}dx\wedge dy
\end{equation}
\end{proof}

\begin{example}The Hess connection of Example \ref{ex:TrivialBiLagrangian} for $n=1$
is the trivial connection $\nabla=d$ whose curvature vanishes identically.
\end{example}

In order to prove Proposition \ref{prop:HessConnection}, let us recall the following formula 
due to M. N. Boyom for Hess connection (see  \cite[p.14]{Boyom1},  \cite[p.360]{Boyom2}, \cite[p.65]{Boyom3} for more details). 

\begin{prop}\label{prop:Boyom} Let $(M,\omega,\F_1,\F_2)$ be a bi-Lagrangian manifold. The Hess connection $\nabla$ of $(\omega,\F_1,\F_2)$ is given as follows
\begin{equation}\label{17c}
\nabla_{(X_1, X_2)}{(Y_1,Y_2)} = (D(X_1,Y_1)+\mathrm{pr}_1([X_2,Y_1]) , D(X_2,Y_2)+\mathrm{pr}_2([X_1,Y_2]))  
\end{equation}
where $D:\,TM\times TM\longmapsto TM$ is the map verifying
\begin{equation}\label{18c} 
i_{D(X,Y )}\omega= L_Xi_Y\omega,
\end{equation} 	
and 	$\mathrm{pr}_j:TM=\F_1\oplus\F_2\to\F_j$ is the  linear projection.
\end{prop}

\begin{proof}[Proof of Proposition \ref{prop:HessConnection}]
We first notice that, in dimension $2$, a symplectic form is just a volume form, and any foliation by curve is automatically Lagrangian.
In order to determine the Hess connexion $\nabla$ of the statement, one promptly compute 
\begin{equation*}
i_{D(\partial_x,\partial_x)}\omega=f_xdy,\ \ \  
i_{D(\partial_y,\partial_y)}\omega=-f_ydx,\ \ \ \text{and}\ \ \ i_{D(\partial_x,0)}\omega=i_{D(0,\partial_y)}\omega=0.
 \end{equation*}
By using (\ref{17c}) and closedness of $\omega$. We deduce 
\begin{equation*}
D(\partial_x,\partial_x)=\frac{f_x}{f}\partial_x,\ \ \  D(\partial_y,\partial_y)=\frac{f_y}{f}\partial_y,\ \ \ \text{and}\ \ \ D(\partial_x,0)=D(0,\partial_y)=0.
\end{equation*}
The statement immediately follows.

Another proof consists in noticing that the connection defined by (\ref{eq:HessMatrixConnection})
satisfies properties of Theorem \ref{b20}, and therefore is the Hess connection by unicity.
\end{proof}

We now apply Corollary \ref{cor:FlatHess} to the bi-Lagrangian manifold $(S_{\bold{a}},\Omega,\F_1,\F_2)$
defined by the foliations $\F_i=\ker(dx_i)$ on the character variety $S_{\bold{a}}=\{P(x_1,x_2,x_3)=0\}$
defined by (\ref{eq:P}) 
$$P(x_1,x_2,x_3)=x_1^2+x_2^2+x_3^2+x_1 x_2 x_3-(a_1 x_1+a_2 x_2+a_3 x_3+a_4)$$
with respect to the symplectic form (\ref{eq:symplectic form})
$$\Omega=\frac{dx_1\wedge dx_2}{2x_3+x_1 x_2-a_3}.$$
Through the projection 
$$S_{\bold{a}}\to \C^2\ ;\ (x_1,x_2,x_3)\mapsto(x_1,x_2)$$
the bi-Lagrangian structure (locally) projects to:
$$(\C^2,\omega, dx_1, dx_2) \ \ \ \text{where}$$
$$\omega=f\cdot dx_1\wedge dx_2 \ \ \ \text{with}\ \ \ f(x_1,x_2)=\frac{1}{2x_3+x_1 x_2-a_3}$$
for a local section $x_3(x_1,x_2)$ of $S_{\bold{a}}\to\C^2$.
Then, we can compute the curvature using the partial derivatives (\ref{eq:partialderivative}) and get
$$\frac{\partial^2f}{\partial x_1\ \partial x_2}\cdot f-\frac{\partial f}{\partial x_1}\cdot \frac{\partial f}{\partial x_2}=$$
$$\frac{a_3\cdot x_1^2x_2^2-4a_1\cdot x_1^2x_2-4a_2\cdot x_1x_2^2+(32-2a_3^2-8a_4)x_1x_2+\cdots}{(2x_3+x_1 x_2-a_3)^6}$$
which is identically vanishing if, and only if, $a_1=a_2=a_3=a_4-4=0$ as can be checked directly on the complete formula.

\section{Other character varieties}

For other Painlevé equations, one can define character varieties (see \cite{vdPutSaito,CMR})
by considering representations of groupoids where Stokes matrices are taken into account 
in the monodromy of the differential equation. 
The explicit list of character varieties has been established by M. van der Put and M.-H. Saito
in \cite[p.2636]{vdPutSaito}, and is given in Table \ref{table:ListCharaterPainleveSvdP}.
\begin{table}[h]
\begin{center}
\begin{tabular}{|c|c|c|}
\hline
type & equation & parameters \\
\hline
$PV$\hfill  & $x_1 x_2 x_3+x_1^2+x_2^2$ & $s_1,s_2\in\C$ \\
& $-(s_1+s_2s_3) x_1-(s_2+s_1s_3) x_2-s_3 x_3$ & $s_3\in\C^*$\\
& $+(s_3^2+s_1s_2s_3+1)=0$ &\\
\hline
$PV^{\mathrm{deg}}$ \hfill & $x_1 x_2 x_3+x_1^2+x_2^2$ & $s_0,s_1\in\C$\\
& $+s_1 x_1+s_2 x_2+1=0$ & \\
\hline
$PIII(D6)$ & $x_1 x_2 x_3+x_1^2+x_2^2$  & $\alpha,\beta\in\C^*$\\
& $+(1+\alpha\beta) x_1+(\alpha+\beta) x_2+\alpha\beta=0$ & \\
\hline
$PIII(D7)$ & $x_1 x_2 x_3+x_1^2+x_2^2+\alpha x_1+ x_2=0$ & $\alpha\in\C^*$\\
\hline
$PIII(D8)$ & $x_1 x_2 x_3+x_1^2-x_2^2-x_1=0$ &\\
\hline
$PIV$\hfill & $x_1 x_2 x_3+x_1^2$ & $s_1\in\C$\\
& $-(s_2^2+s_1s_2) x_1-s_2^2 x_2-s_2^2 x_3$ & $s_2\in\C^*$\\
& $+(s_2^2+s_1s_2^3)=0$ & \\
\hline
$PII^{FN}$\hfill & $x_1 x_2 x_3+x_1-x_2+x_3+s=0$ & $s\in\C$\\
\hline
$PII$\hfill & $x_1 x_2 x_3-x_1-\alpha x_2-x_3+\alpha+1=0$ & $\alpha\in\C^*$\\
\hline
$PI$\hfill & $x_1 x_2 x_3+x_1+x_2+1=0$ & \\
\hline
\end{tabular}
\end{center}
\label{table:ListCharaterPainleveSvdP}
\end{table}

Each polynomial $P$ in the list defines a cubic surface which represents monodromy space 
for a Painlevé equation (named $PV,\cdots,PI$). This surface admits a natural symplectic structure
defined by P. Boalch in \cite{Boalch}, and which writes
\begin{equation}\label{eq:SymplecticPVtoPI}
\Omega=\frac{dx_1\wedge dx_2}{P_{x_3}}=\frac{dx_2\wedge dx_3}{P_{x_1}}
=\frac{dx_3\wedge dx_1}{P_{x_2}}
\end{equation}
with standard notation $P_{x_i}=\frac{\partial P}{\partial x_i}$.

For each of these polynomials $P$, we compute the curvature of the $3$-web $(\F_1,\F_2,\F_3)$, as well as the curvature
of the Hess connection associated to $(P=0,\Omega,\F_i,\F_j)$ for $(i,j)=(1,2),(2,3),(1,3)$. Then, with similar computations
as in Sections (\ref{sec:Webs}) and (\ref{sec:Hess}), we obtain, after projection on $(x_1,x_2)$-plane say, a curvature
$2$-form $F(x_1,x_2,x_3)dx_1\wedge dx_2$ with $F$ rational with respect to $(x_1,x_2,x_3)$.
We can moreover reduce $F$ with the relation $P=0$ so that it is linear with respect to $x_3$.
We thus get an expression $F(\underline{s},x_1,x_2)$, which is rational in $(x_1,x_2)$ with parameters $\underline{s}$, 
and which can be identically zero or not.
In the former case, we promptly deduce that the curvature is identically zero and the $3$-web is parallelizable.
In the latter case, then the coefficients depending on $\underline{s}$ in the numerator of $F$ define an ideal $\mathcal I=\mathcal I(\underline{s})$ whose zero-set is the set of parameters $\underline{s}$ for which the curvature vanishes identically.
In Table \ref{table:curvature}, we provide the ideal of coefficients of the numerator of the curvature of the web (column 2)
and of Hess connection for each pair of foliations (column 3-5) for each family of polynomials (column 1).

\begin{table}[h]
\begin{center}
\begin{tabular}{|c|c|c|c|c|}
\hline 
Name & $\mathcal W$ & $(\F_1,\F_2)$ & $(\F_2,\F_3)$ & $(\F_1,\F_3)$ \\
\hline\hline
\small $PV$ & $(1)$ &  $(s_3)$ & $(1)$ & $(1)$ \\
\hline
\small $PV^{\mathrm{deg}}$ &$(1)$ &   $0$ & $(1)$ & $(1)$ \\
\hline
\small $PIII(D6)$ &$(1)$ &   $0$ & $(1)$ & $(1)$ \\
\hline
\small $PIII(D7)$ &$(1)$ &   $0$ & $(1)$ & $(1)$ \\
\hline
\small $PIII(D8)$ &$(1)$ &   $(1)$ & $(1)$ & $0$ \\
\hline
\small $PIV$ & $(s_2^2)$ &   $(s_2)^2$ & $(s_1s_2,s_2^2)$ & $(s_2^2)$ \\
\hline
\small $PII-FN$ &$0$ &   $(1)$ & $(1)$ & $(1)$ \\
\hline
\small $PII$ &$0$ &   $(1)$ & $(1)$ & $(\alpha)$ \\
\hline
\small $PI$ &$0$ &   $0$ & $(1)$ & $(1)$ \\
\hline
\end{tabular}
\end{center}
\label{table:curvature}
\end{table}

For instance, in the Painlevé I case, $S=\{x_1 x_2 x_3+x_1+x_2+1=0\}$, the $3$-web projects to 
$$\F_1=\ker(dx_1),\ \ \ \F_2=\ker(dx_2)\ \ \ \text{and}\ \ \ \F_3=\ker\left(\frac{dx_1}{x_1(x_1+1)}+\frac{dx_2}{x_2(x_2+1)}\right)$$
which is flat. On the other hand, the curvature of the Hess connection associated to 
$$\left(S,\frac{dx_1}{x_1}\wedge \frac{dx_2}{x_2},\F_1,\F_2\right)$$
is given by $\frac{dx_1\wedge dx_2}{(x_1x_2)^2}$.

\end{document}